\documentclass{amsart}

\usepackage{amssymb}
\usepackage{amsmath}
\usepackage{amscd}
\usepackage{psfrag}
\usepackage{graphicx} 

\newcommand{\C}{\mathbf{C}}
\newcommand{\R}{\mathbf{R}}
\newcommand{\Z}{\mathbf{Z}}

\newcommand{\transp}{\mathrm{\, }^t\hskip-.07cm}

\theoremstyle{plain}
\newtheorem{theorem}{Theorem}[section]

\newtheorem{proposition}[theorem]{Proposition}
\newtheorem{lemma}[theorem]{Lemma}

\theoremstyle{definition}
\newtheorem{definition}[theorem]{Definition}

\theoremstyle{remark}
\newtheorem{remark}[theorem]{Remark}

\newtheorem{example}[theorem]{Example}

\makeatletter
 
 \@addtoreset{equation}{section}
\makeatother
\title{Cobordism of algebraic knots defined by
Brieskorn polynomials}
\author{Vincent Blanl\oe il}
\address{D\'epartement de Math\'ematiques,
Universit\'e de Strasbourg,
7 rue  Ren\'e Descartes,
67084 Strasbourg cedex, France}
\email{blanloeil@math.u-strasbg.fr}

\author{Osamu Saeki}
\thanks{The second
author has been partially
supported by Grant-in-Aid for Scientific Research
(No.~19340018), Japan Society for the Promotion of Science.
He has also been partially supported by
the Louis Pasteur University of Strasbourg, France,
and by FY 2008 
Researcher Exchange Program between JSPS and CNRS}
\address{Faculty of Mathematics, Kyushu University,
Hakozaki, Fukuoka 812-8581, Japan}
\email{saeki@math.kyushu-u.ac.jp}



\date{\today}

\begin{document}

\begin{abstract} 
In this paper we study the cobordism of algebraic knots 
associated with weighted homogeneous polynomials, 
and in particular Brieskorn polynomials.
Under some
assumptions we prove that the associated algebraic
knots are cobordant 
if and only if the Brieskorn polynomials have 
the same exponents.
\end{abstract}

\maketitle 

\section{Introduction}\label{section1}

A \emph{Brieskorn polynomial} is a polynomial of the form
$$P(z) = z_1^{a_1} + z_2^{a_2} + \cdots + z_{n+1}^{a_{n+1}}$$
with $z = (z_1, z_2, \ldots, z_{n+1})$, $n \geq 1$, where
the integers $a_j \geq 2$, $j = 1, 2, \ldots, n+1$, 
are called the \emph{exponents}.
The complex hypersurface in $\C^{n+1}$ defined by $P=0$
has an isolated singularity at the origin, which is
called a \emph{Brieskorn singularity}.

In this paper, we will study Brieskorn singularities 
up to cobordism.
We prove that two Brieskorn singularities
have cobordant algebraic knots if and only if 
they have the same set of exponents, provided that no exponent 
is a multiple of another for each of the two Brieskorn polynomials. 
Consequently, for such Brieskorn polynomials the multiplicity is 
an invariant of the cobordism class of the associated 
algebraic knot. 

To be more precise, 
let $f : (\C^{n+1}, 0) \to (\C, 0)$ be a holomorphic function germ
with an isolated critical point at the origin.
We denote by $D^{2n+2}_{\varepsilon}$ the closed ball of radius 
$\varepsilon > 0$ centred at $0$ in $\C^{n+1}$, and by 
$S^{2n+1}_{\varepsilon}$ its boundary. According to Milnor
\cite{Milnor}, the oriented homeomorphism
class of the pair $(D^{2n+2}_{\varepsilon}, 
f^{-1}(0) \cap D^{2n+2}_{\varepsilon})$ does not depend on the
choice of a sufficiently small $\varepsilon > 0$, and
by definition it is the \emph{topological type} of $f$.
(For other equivalent definitions, we
refer the reader to \cite{King, Perron, Saeki89}.)
The oriented diffeomorphism class of the pair 
$(S^{2n+1}_{\varepsilon}, K_f)$, with
$K_f = f^{-1}(0) \cap S^{2n+1}_{\varepsilon}$, is the 
\emph{algebraic knot} associated with $f$, where
$K_f$ is a closed oriented $(2n-1)$-dimensional manifold.
According to Milnor's cone structure theorem
\cite{Milnor}, the algebraic knot $K_f$ determines the 
topological type of $f$. In fact, it is known
that the converse also holds.

\begin{definition}\label{dfn:cob}
An \emph{$m$-dimensional knot} (\emph{$m$-knot}, for short)
is a closed oriented $m$-dimensional submanifold
of the oriented $(m+2)$-dimensional sphere $S^{m+2}$.
Two $m$-knots $K_0$ and $K_1$ in 
$S^{m+2}$ are said to be \emph{cobordant} if there exists a 
properly embedded oriented $(m+1)$-dimensional
submanifold $X$ of $S^{m+2} \times [0,1]$ such that
\begin{enumerate}
\item $X$ is diffeomorphic to $K_0 \times [0,1]$, and
\item $\partial X = (K_0 \times \{0\}) \cup (-K_1 \times \{1\}).$
\end{enumerate}
Such a manifold $X$ is called a \emph{cobordism}
between $K_0$ and $K_1$ (see Fig.~\ref{fig2}).
\end{definition} 

\begin{figure}[tb]
\centering
 \begin{picture}(100,120)(0,0)
 \thicklines
 \qbezier(10,90)(30,50)(15,10)
 \put(21,57){\circle*{3}}
  \put(22,40){\circle*{3}}
 \put(8,48){\makebox{\small$K_{0}$}}
  \put(-5,100){\makebox{\small $ S^{m+2}\times\{0\}$}}
 \thinlines
 \qbezier[30](15,90)(35,50)(20,10)
 \qbezier[30](20,90)(40,50)(25,10)
 \qbezier[30](25,90)(45,50)(30,10)
 \qbezier[30](30,90)(50,50)(35,10)
 \qbezier[30](35,90)(55,50)(40,10)
 \qbezier[30](40,90)(60,50)(45,10)
 \qbezier[30](45,90)(65,50)(50,10)
 \qbezier[30](50,90)(70,50)(55,10)
 \qbezier[30](55,90)(75,50)(60,10)
 \qbezier[30](60,90)(80,50)(65,10)
 \qbezier[30](65,90)(85,50)(70,10)
 \qbezier[30](70,90)(90,50)(75,10)
 \qbezier[30](75,90)(95,50)(80,10)
 \qbezier[30](80,90)(100,50)(85,10)
 \qbezier[30](85,90)(105,50)(90,10)
 \qbezier(21,57)(95,35)(65,55)
 \qbezier(65,55)(35,75)(101,57)
 \qbezier(22,40)(35,20)(102,40)
 \thicklines
 \qbezier(90,90)(110,50)(95,10)
 \put(101,57){\circle*{3}}
 \put(102,40){\circle*{3}}
 \put(104,48){\makebox{\small$K_{1}$}}
 \put(85,100){\makebox{\small $ S^{m+2}\times\{1\}$}}
 \put(35,0){\makebox{\small $ S^{m+2}\times[0, 1]$}}
 \end{picture}
\caption{A cobordism between $K_0$ and $K_1$}
\label{fig2}
\end{figure}
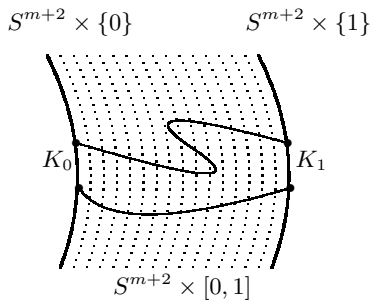

In \cite{BM}, for $n \geq 3$, necessary and sufficient
conditions for two algebraic $(2n-1)$-knots to be cobordant
have been obtained in terms of Seifert forms (for the
definition of the Seifert form, see \S\ref{section2}).
However, the computation of the Seifert form
of a given algebraic knot is very difficult,
and an explicit calculation 
is known only for a very limited class of algebraic
knots. (In fact, even for algebraic knots associated
with weighted homogeneous polynomials, Seifert forms
have not been determined yet, as far as the authors know.)
Furthermore, even if we know 
the Seifert forms explicitly, it is still difficult to see 
if given two such forms satisfy the 
algebraic conditions given in \cite{BM} or not.
So, it is worthwhile to study the conditions
for two algebraic knots associated with
weighted homogeneous polynomials to be cobordant.
We note that cobordism does not necessarily
imply isotopy for algebraic knots in general. For
details, see the survey article \cite{BS}.

It is known that cobordant algebraic knots
have Witt equivalent Seifert forms (for
details, see \S\ref{section2}).
In this paper, we give a necessary and sufficient
condition for two algebraic knots associated
with weighted homogeneous polynomials to have Witt
equivalent Seifert forms over the real numbers
in terms of their weights. Using this result,
we give some conditions for two algebraic
knots associated with Brieskorn polynomials to
be cobordant in terms of the exponents.
Under some assumptions, we show that two
such knots are cobordant if and only if
the Brieskorn polynomials have the same set of exponents.

The paper is organized as follows. In \S\ref{section2},
we state our results. We give a necessary and
sufficient condition for two nondegenerate
weighted homogeneous polynomials to have
Witt equivalent Seifert forms over the real
numbers, in terms of their weights.
Then, we give more explicit results
for Brieskorn polynomials. In \S\ref{section3},
we prove the results stated in \S\ref{section2}.
In \S\ref{section4}, we give more
precise results in the case of two and
three variables.

Throughout the paper we work in the smooth category.
All the homology groups are with integer coefficients
unless otherwise specified.

\section{Results}\label{section2}

Let $f(z)$ be a polynomial in $\C^{n+1}$ with
an isolated critical point at the origin.
We denote by $F_f$ the \emph{Milnor fiber} associated with
$f$, i.e., $F_f$ is the closure of a fiber
of the Milnor fibration 
$\varphi_f : S^{2n+1}_\varepsilon \setminus K_f \to S^1$ 
defined by $\varphi_f(z) = f(z)/|f(z)|$. 
According to Milnor \cite{Milnor}, $F_f$ is a compact
$2n$-dimensional submanifold of $S^{2n+1}_\varepsilon$
which is homotopy equivalent to the bouquet of a finite
number of copies of the $n$-dimensional sphere.

The Seifert form 
$$L_f : H_n(F_f) \times H_n(F_f) \to \Z$$
associated with $f$ is defined by
$$L_f(\alpha, \beta) = \mathrm{lk}(a_+, b),$$
where $a$ and $b$ are $n$-cycles representing
$\alpha$ and $\beta$ in $H_n(F_f)$ respectively,
$a_+$ is the $n$-cycle in $S^{2n+1}_\varepsilon$ obtained
by pushing $a$ into the positive normal direction of
$F_f$, and $\mathrm{lk}$ denotes 
the linking number of $n$-cycles in $S^{2n+1}_{\varepsilon}$.
It is known that the isomorphism class of the
Seifert form is a topological invariant of $f$.
Furthermore, two algebraic knots $K_f$ and $K_g$
associated with polynomials $f$ and $g$ in $\C^{n+1}$,
respectively, with isolated critical points at the origin
are isotopic in $S^{2n+1}_\varepsilon$ if and only if
their Seifert forms $L_f$ and $L_g$ are isomorphic,
provided that $n \geq 3$.

In fact, 
algebraic knots are \emph{simple fibered} knots as follows.
We say that an oriented $m$-knot $K$ is \emph{fibered}
if there exists a smooth
fibration $\phi : S^{m+2} \setminus K \to S^1$ 
and a trivialization $\tau : N_K \to K \times D^2$
of a closed tubular neighborhood $N_K$ of $K$
in $S^{m+2}$ such that $\phi|_{N_K \setminus K}$
coincides with $\pi \circ \tau|_{N_K \setminus K}$,
where $\pi : K \times (D^2 \setminus \{0\}) \to
S^1$ is the composition of the projection to the
second factor and the obvious projection
$D^2 \setminus \{0\} \to S^1$.
Note that then the closure of each fiber of $\phi$
in $S^{m+2}$
is a compact $(m+1)$-dimensional oriented manifold
whose boundary coincides with $K$.
We shall often call the 
closure of each fiber simply a \emph{fiber}.
Moreover, for $m=2n-1 \geq 1$ we say that a fibered $(2n-1)$-knot 
$K$ is \emph{simple} if each fiber of $\phi$ is 
$(n-1)$-connected and $K$ is $(n-2)$-connected.
For details we refer the reader to \cite{BS}.
Note that two simple fibered
$(2n-1)$-knots are isotopic if and only if
they have isomorphic Seifert forms, provided
$n \geq 3$ (see \cite{Durfee, Kato}).

\begin{definition}
Two bilinear forms $L_i : G_i \times G_i \to \Z$,
$i = 0, 1$, defined on free abelian groups $G_i$ of
finite ranks are said to be \emph{Witt equivalent}
if there exists a direct summand $M$ of $G_0 \oplus
G_1$ such that $(L_0 \oplus (-L_1))(x, y) = 0$
for all $x, y \in M$ and twice the rank of $M$
is equal to the rank of $G_0 \oplus G_1$.
In this case, $M$ is called a \emph{metabolizer}.

Furthermore, we say that $L_0$ and $L_1$
are \emph{Witt equivalent over the real numbers}
if there exists a vector subspace $M_{\R}$
of $(G_0 \otimes \R) \oplus (G_1 \otimes \R)$
such that $(L_0^\R \oplus (-L_1^\R))(x, y)
= 0$ for all $x, y \in M_{\R}$ and
$2\dim_\R M_\R = \dim_\R (G_0 \otimes \R)
+ \dim_\R (G_1 \otimes \R)$, where
$L_i^\R : (G_i \otimes \R) \times (G_i \otimes \R)
\to \R$ is the real bilinear form associated with $L_i$,
$i = 0, 1$.
\end{definition}

The following lemma is well-known
(for example, see \cite{BM}).

\begin{lemma}
If two simple fibered $(2n-1)$-knots are
cobordant, then their Seifert forms
are Witt equivalent. In particular,
they are Witt equivalent over the real numbers
as well.
\end{lemma}

Now, let $f$ be a 
\emph{weighted homogeneous polynomial} in $\C^{n+1}$,
i.e., there exist positive rational numbers
$(w_1, w_2, \ldots, w_{n+1})$, called
\emph{weights}, such that
for each monomial $c z_1^{k_1} z_2^{k_2} \cdots z_{n+1}^{k_{n+1}}$,
$c \neq 0$, of $f$, we have
$$\sum_{j=1}^{n+1} \frac{k_j}{w_j} = 1.$$
We say that $f$ is \emph{nondegenerate} if
it has an isolated critical point at the origin.
Saito \cite{Saito} has shown that if $f$ is nondegenerate, then
by an analytic change of coordinate, 
$f$ can be transformed to a nondegenerate
weighted homogeneous polynomial such that all
the weights are greater than or equal to $2$. 
Furthermore, under the assumption that the
weights are all greater than or equal to $2$,
the weights are analytic invariants of the polynomial.

Let $f$ be a nondegenerate weighted homogeneous
polynomial in $\C^{n+1}$ with weights
$(w_1, w_2, \ldots, w_{n+1})$ such that $w_j \geq 2$ for
all $j$. Set
$$P_f(t) = \prod_{j=1}^{n+1} \frac{t-t^{1/w_j}}{t^{1/w_j}-1}.$$
Note that $P_f(t)$ is a polynomial in $t^{1/m}$ 
over $\Z$ for
some positive integer $m$. It is known
that two nondegenerate weighted homogeneous
polynomials $f$ and $g$ in $\C^{n+1}$
have the same weights if and only if
$P_f(t) = P_g(t)$ (see \cite{Steenbrink}).

Our first result is the following.

\begin{theorem}\label{thm1}
Let $f$ and $g$ be nondegenerate
weighted homogeneous polynomials in $\C^{n+1}$.
Then, their Seifert forms are Witt equivalent
over the real numbers if and only if
$P_f(t) \equiv P_g(t) \mod{t+1}$.
\end{theorem}

\begin{remark}
The above theorem should be compared with
the result, obtained in \cite{Saeki2000}, which states that
the Seifert forms associated with nondegenerate
weighted homogeneous polynomials $f$ and $g$
are isomorphic over the real numbers if
and only if 
$P_f(t) \equiv P_g(t) \mod{t^2-1}$.
\end{remark}

Let us now consider the case of Brieskorn polynomials.
Note that a Brieskorn polynomial is always a nondegenerate
weighted homogeneous polynomial and its weights
coincide with its exponents.

\begin{proposition}\label{prop:cot}
Let
$$
f(z) = \sum_{j=1}^{n+1} z_j^{a_j} \quad \text{and} \quad
g(z) = \sum_{j=1}^{n+1} z_j^{b_j}
$$
be Brieskorn polynomials. 
Then, their Seifert forms are Witt equivalent
over the real numbers if and only if
\begin{equation}
\prod_{j=1}^{n+1} \cot \frac{\pi \ell}{2a_j}
= \prod_{j=1}^{n+1} \cot \frac{\pi \ell}{2b_j}
\label{eq:cot}
\end{equation}
holds for all odd integer $\ell$.
\end{proposition}

To each polynomial $Q(t) = \prod_{j=1}^k (t-\alpha_j)$, 
with $\alpha_1, \alpha_2, \ldots, \alpha_k$ in $\C^\ast$, 
the multiplicative group of nonzero complex numbers, set
$$\mathrm{divisor}\,Q(t) = \langle \alpha_1 \rangle + 
\langle \alpha_2 \rangle + \cdots + \langle \alpha_k \rangle,$$
which is regarded as an element of the integral group ring 
$\Z\C^\ast$ and is called the \emph{divisor} of $Q$.
For a positive integer $a$,
set $\Lambda_a = \mathrm{divisor}\,(t^a-1)$. 
For the notation and some properties
of $\Lambda_a$, we refer the reader to \cite{MO}.

Let $f$ be a nondegenerate weighted homogeneous
polynomial in $\C^{n+1}$ with weights
$(w_1, w_2, \ldots, w_{n+1})$ such that
$w_j \geq 2$ for all $j$.
Let $\Delta_f(t)$ be the characteristic
polynomial of the monodromy of $f$ 
(see \cite{Milnor}).
Then, by Milnor-Orlik \cite{MO}, we have
\begin{equation}\label{Milnor-Orlik}
\mathrm{divisor}\,\Delta_f(t) = \prod_{j=1}^{n+1}
\left(\frac{1}{v_j}\Lambda_{u_j}-1\right),
\end{equation}
where $w_j = u_j/v_j$, and $u_j$ and $v_j$
are relatively prime positive integers,
$j = 1, 2, \ldots, n+1$.
In the case of a Brieskorn polynomial, by virtue of
the Brieskorn-Pham theorem (for example,
see \cite{Milnor}), we have 
$$
\mathrm{divisor}\,\Delta_f(t) = \prod_{j=1}^{n+1}
(\Lambda_{a_j}-1),
$$
which can also be deduced from the Milnor-Orlik
theorem mentioned above.

\begin{proposition}\label{prop:FM}
$(1)$ Let $f$ and $g$ be nondegenerate weighted
homogeneous polynomials in $\C^{n+1}$
with weights 
$$(u_1/v_1, u_2/v_2, \ldots, u_{n+1}/v_{n+1})
\quad \mbox{and} \quad
(u'_1/v'_1, u'_2/v'_2, \ldots, u'_{n+1}/v'_{n+1})$$
respectively, where $u_j$ and $v_j$ \textup{(}resp.\ $u'_j$
and $v'_j$\textup{)} are relatively prime positive
integers, $j = 1, 2, \ldots, n+1$.
If their Seifert forms are Witt equivalent
over the real numbers, then we have
$$\prod_{j=1}^{n+1}
\left(\frac{1}{v_j}\Lambda_{u_j}-1\right)
\equiv \prod_{j=1}^{n+1}
\left(\frac{1}{v'_j}\Lambda_{u'_j}-1\right)
\pmod{2}.$$

$(2)$ Let $f$ and $g$ be Brieskorn polynomials
as in Proposition~\textup{\ref{prop:cot}}.
If their Seifert forms are Witt equivalent
over the real numbers, then we have
$$\prod_{j=1}^{n+1}(\Lambda_{a_j}-1) \equiv 
\prod_{j=1}^{n+1}(\Lambda_{b_j}-1) \pmod{2}.$$
\end{proposition}

The following theorem partially answers \cite[Problem~11.10]{BS}
in the positive.

\begin{theorem}\label{thm2}
Suppose that for each of the Brieskorn polynomials 
$$
f(z) = \sum_{j=1}^{n+1} z_j^{a_j} \quad \text{and} \quad
g(z) = \sum_{j=1}^{n+1} z_j^{b_j}
$$
no exponent is a multiple of another one. Then,
the knots $K_f$ and $K_g$ are cobordant if and only if 
$a_j = b_j$, $j = 1, 2, \ldots, n+1$, up to order.
\end{theorem}

Concerning \cite[Problem~11.9]{BS}, we have the
following. Recall that the multiplicity of a
Brieskorn polynomial coincides with the smallest
exponent.

\begin{proposition}\label{prop:multi}
Suppose that for each of the Brieskorn polynomials
$$
f(z) = \sum_{j=1}^{n+1} z_j^{a_j} \quad \text{and} \quad
g(z) = \sum_{j=1}^{n+1} z_j^{b_j}
$$
the exponents are pairwise distinct. If $K_f$ and $K_g$
are cobordant, then the multiplicities of $f$ and $g$ coincide.
\end{proposition}

\section{Proofs}\label{section3}

In this section, we prove the results stated
in \S\ref{section2}.

\begin{proof}[Proof of Theorem~\textup{\ref{thm1}}] 
It is known that the Seifert form
associated with the polynomial
$$\tilde{f}(z_1, z_2, \ldots, z_{n+2})
= f(z_1, z_2, \ldots, z_{n+1}) + z_{n+2}^2$$
is naturally isomorphic to $(-1)^{n+1}L_f$
(for example, see \cite{Sakamoto}
or \cite[Lemma~2.1]{Saeki2000}).
Furthermore, we have $P_{\tilde{f}}(t) = t^{1/2}P_f(t)$.
Hence,
by considering $f(z) + z_{n+2}^2$ and $g(z) + z_{n+2}^2$
if necessary, we may assume that $n$ is even.



Recall that
$$
H^n(F_f; \C) = \oplus_\lambda H^n(F_f; \C)_\lambda,
$$
where $F_f$ is the Milnor fiber for $f$,
$\lambda$ runs over all
the roots of the characteristic polynomial $\Delta_f(t)$, and
$H^n(F_f; \C)_\lambda$ is the eigenspace of the
monodromy $H^n(F_f; \C) \to H^n(F_f; \C)$ corresponding to the
eigenvalue $\lambda$. It is easy to see that the intersection form
$S_f = L_f + \transp{L_f}$ of $F_f$ on $H^n(F_f; \C)$
decomposes as the orthogonal direct sum of 
$(S_f)|_{H^n(F_f; \C)_\lambda}$.
Let 
$\mu(f)_\lambda^+$ (resp.\ $\mu(f)_\lambda^-$) denote the 
number of positive (resp.\ negative) eigenvalues 
of $(S_f)|_{H^n(F; \C)_\lambda}$. 
The integer
$$
\sigma_\lambda(f) = \mu(f)_\lambda^+ - \mu(f)_\lambda^-,
$$
is called the \emph{equivariant signature} of 
$f$ with respect to $\lambda$ (for details,
see \cite{Neumann, SSS}).
According to Steenbrink \cite{Steenbrink2}, 
putting $P_f(t) = \sum c_\alpha t^\alpha$,
we have
$$
\sigma_\lambda(f) =  
\sum_{\substack{\lambda = \exp(-2\pi i\alpha)
\\[1pt] \lfloor
\alpha \rfloor: \mbox{ \rm \scriptsize even}}}
c_{\alpha}  
- \sum_{\substack{\lambda = \exp(-2\pi i\alpha), 
\\[1pt] \lfloor
\alpha \rfloor: \mbox{ \rm \scriptsize odd}}} c_{\alpha}
$$
for $\lambda \neq 1$,
where $i = \sqrt{-1}$, and $\lfloor \alpha \rfloor$ is the largest integer
not exceeding $\alpha$.

Now, suppose that the Seifert forms $L_f$ and $L_g$
are Witt equivalent over the real numbers.
Then, the equivariant signatures $\sigma_\lambda(f)$ and 
$\sigma_\lambda(g)$ coincide 
for all $\lambda$ 
(for example, see \cite{DH}.
See also \cite{L1, L2} for the spherical knot case).
Note that by \cite[Lemma~2.3]{Saeki2000},
the equivariant signature for $\lambda = 1$
is always equal to zero.

Set $P_f(t) = P_f^0(t) + P_f^1(t)$, where
$P_f^0(t)$ (resp.\ $P_f^1(t)$) is the sum of those
terms $c_\alpha t^\alpha$ with $\lfloor \alpha \rfloor \equiv 0 \pmod{2}$
(resp.\ $\lfloor \alpha \rfloor \equiv 1 \pmod{2}$). 
We define $P_g^0(t)$ and $P_g^1(t)$ similarly.
Since the equivariant signatures of $f$ and $g$
coincide, we have
$$tP_f^0(t) - P_f^1(t) \equiv tP_g^0(t) - P_g^1(t)
\mod{t^2-1}$$
and
$$tP_f^1(t) - P_f^0(t) \equiv tP_g^1(t) - P_g^0(t)
\mod{t^2-1}$$
(for details, see \cite{N, Saeki2000}).
Adding up these two congruences we see that
\begin{equation}
(t-1)P_f(t) \equiv (t-1)P_g(t) \mod{t^2-1},
\label{eq:t2-1}
\end{equation}
which implies that 
\begin{equation}
P_f(t) \equiv P_g(t) \mod{t+1}.
\label{eq:t+1}
\end{equation}

Conversely, suppose that (\ref{eq:t+1}) holds.
Then, we have (\ref{eq:t2-1}), which implies that
the Seifert forms $L_f$ and $L_g$ have the same
equivariant signatures. Then, we see that they
are Witt equivalent over the real numbers
by virtue of \cite[\S4]{Saeki2000}.
This completes the proof.
\end{proof}

Note that $P_f(t)$ and $P_g(t)$ are
polynomials in $s = t^{1/m}$ for some $m$.
Let us put $Q_f(s) = P_f(t)$ and $Q_g(s) = P_g(t)$.
Then, it is easy to see that
(\ref{eq:t+1}) holds if and only if
$Q_f(\xi) = Q_g(\xi)$ for all $\xi$ with
$\xi^m = -1$.
Note that $\xi$ is of the form
$\exp(\pi \sqrt{-1} \ell/m)$ with $\ell$ odd
and that
$$\frac{-1-\exp(\pi \sqrt{-1} \ell/a_j)}{\exp(\pi \sqrt{-1} \ell/a_j)-1}
= \sqrt{-1} \cot \frac{\pi \ell}{2a_j}.$$
Then, we immediately get Proposition~\ref{prop:cot}.

By considering those odd integers $\ell$
which gives zero in (\ref{eq:cot}),
we get the following.

\begin{proposition}\label{proposition:odd}
Let $f$ and $g$ be the Brieskorn polynomials 
$$
f(z) = \sum_{j=1}^{n+1} z_j^{a_j} \quad \text{and} \quad
g(z) = \sum_{j=1}^{n+1} z_j^{b_j}.
$$
If their Seifert forms are Witt equivalent
over the real numbers, then we have
\begin{eqnarray*}
& & \{\ell \in \Z\,|\, \mbox{\rm $\ell$
is odd and is a multiple of some $a_j$}\} \\
& = & 
\{\ell \in \Z\,|\, \mbox{\rm $\ell$
is odd and is a multiple of some $b_j$}\}.
\end{eqnarray*}
In particular, if $a_j$ is odd for some $j$,
then $b_k$ is odd for some $k$,
and the minimal odd exponent for $f$
coincides with that for $g$.
\end{proposition}

\begin{remark}
For nondegenerate weighted homogeneous polynomials,
we also have results similar to 
Propositions~\ref{prop:cot} or \ref{proposition:odd}. 
However, the statement becomes complicated, so we omit
them here (compare this with \cite[Proposition~2.6]{Saeki2000}).
\end{remark}

Now, Proposition~\ref{prop:FM} is a consequence
of the Milnor-Orlik and
Brieskorn-Pham theorems on the characteristic
polynomials \cite{Milnor, MO} together 
with the Fox-Milnor type relation. Here, a Fox-Milnor type 
relation for two polynomials $f$ and $g$ with Witt equivalent 
Seifert forms means that there exists a polynomial 
$\gamma(t)$ such that $\Delta_f(t)\,\Delta_g(t) 
= \pm t^{\deg(\gamma)}\gamma(t)\,\gamma(t^{-1})$ (for details,
see \cite{BS}, for example). Here we give another
proof, using Theorem~\ref{thm1}, as follows.

\begin{proof}[Proof of Proposition~\textup{\ref{prop:FM}}]
Since $P_f(t) \equiv P_g(t)
\mod{t+1}$, there exists a polynomial
$R(t) \in \Z[t^{1/m}]$ for some $m$ such that
$$P_f(t) - P_g(t) = (t+1)R(t) = (t-1)R(t) + 2R(t).$$
Therefore, for each $\lambda \in S^1$,
the multiplicities of $\lambda$ in the characteristic
polynomials $\Delta_f(t)$ and $\Delta_g(t)$
are congruent modulo $2$ to each other
(for details, see \cite{N, Saeki2000}, for example). 
Then, the result follows in view of the 
Milnor-Orlik formula (\ref{Milnor-Orlik}) for the characteristic
polynomial.
\end{proof}

For the proof of Theorem~\ref{thm2}, we
need the following.

\begin{lemma}\label{lemma:lambda}
For integers $2 \leq a_1 < a_2 < \cdots < a_p$ and 
$2 \leq b_1 < b_2 < \cdots < b_q$, we have
\begin{equation}
\sum_{j = 1}^p \Lambda_{a_j} \equiv
\sum_{j=1}^q \Lambda_{b_j} \pmod{2}
\label{eq:lambda}
\end{equation}
if and only if $p=q$ and $a_j = b_j$ for all
$j$.
\end{lemma}

\begin{proof}
Suppose that $a_p < b_q$. Then the coefficient
of $\exp(2 \pi \sqrt{-1}/b_q)$ on the right hand side
of (\ref{eq:lambda}) is equal to $1$, while the
corresponding coefficient on the left hand side
is equal to $0$. This is a contradiction.
So, we must have $a_p = b_q$. Then we have
$$\sum_{j = 1}^{p-1} \Lambda_{a_j} \equiv
\sum_{j=1}^{q-1} \Lambda_{b_j} \pmod{2}.$$
Therefore, by induction, we get the desired conclusion.
\end{proof}

\begin{proof}[Proof of Theorem~\textup{\ref{thm2}}]
Suppose that the algebraic knots $K_f$ and $K_g$
are cobordant.
We may assume $a_1 < a_2 < \cdots < a_{n+1}$
and $b_1 < b_2 < \cdots < b_{n+1}$.
By Proposition~\ref{prop:FM} (2), we have
\begin{equation}
\prod_{j=1}^{n+1}(\Lambda_{a_j}-1) - (-1)^{n+1} \equiv 
\prod_{j=1}^{n+1}(\Lambda_{b_j}-1) - (-1)^{n+1} \pmod{2}.
\label{eq:eq1}
\end{equation}

Recall that for positive integers $a$ and $b$, we have
$$\Lambda_a \Lambda_b = (a, b) \Lambda_{[a, b]},$$
where $(a, b)$ is the greatest common divisor of
$a$ and $b$, and $[a, b]$ denotes the least
common multiple of $a$ and $b$.

By considering
the term of the form $\Lambda_d$ with the smallest $d$
on both sides of (\ref{eq:eq1}),
we see that $a_1 = b_1$ by Lemma~\ref{lemma:lambda}. 
By subtracting $\Lambda_{a_1}$
from the both sides of (\ref{eq:eq1}), we see $a_2 = b_2$, since
$a_2$ (or $b_2$) is not a multiple of $a_1$
(resp.\ $b_1$).
Then, by further subtracting $\Lambda_{a_2} + (a_1, a_2)
\Lambda_{[a_1, a_2]}$ from (\ref{eq:eq1}),
we see $a_3 = b_3$, since $a_3$ (or $b_3$)
is not a multiple of $a_1$ or $a_2$ (resp.\ $b_1$
or $b_2$). Repeating this procedure, we see that
$a_j = b_j$ for all $j$.

Conversely, if $f$ and $g$ have the same set of
exponents, then $K_f$ and $K_g$ are isotopic
and hence cobordant. This completes the proof.
\end{proof}

\begin{proof}[Proof of Proposition~\textup{\ref{prop:multi}}]
In the proof of Theorem~\ref{thm2}, we proved that the smallest 
exponents of $f$ and $g$ are equal, provided that
there is only one smallest exponent for
each of $f$ and $g$. Since we assume that
the exponents of $f$ (or $g$)
are pairwise distinct, the same argument works. 
\end{proof}

\begin{remark}
Theorem~\ref{thm2} implies that
two algebraic knots $K_f$ and $K_g$ 
associated with certain Brieskorn polynomials
are isotopic if and only of they
are cobordant. Recall that according to
Yoshinaga-Suzuki \cite{YS}, two algebraic
knots associated with Brieskorn polynomials
in general are isotopic if and
only if they have the same set of exponents. In fact,
they showed that the characteristic polynomials
coincide if and only if the Brieskorn polynomials
have the same set of exponents. 
\end{remark}

\begin{remark}
For the case where $n = 2$ and the knots are
homology spheres, Theorem~\ref{thm2} has been obtained
in \cite{Saeki88} by using the Fox-Milnor type
relation.
\end{remark}

\begin{example}
For all integers $p_1, p_2, \ldots, p_{n-3} \geq 2$,
$n \geq 3$, the product of the 
characteristic polynomials of the algebraic knots associated 
with 
$$f(z) = z_1^{p_1} + z_2^{p_2} + \cdots + z_{n-3}^{p_{n-3}} + 
z_{n-2}^8 + z_{n-1}^8 + z_n^4 + z_{n+1}^4$$ 
and
$$g(z) = z_1^{p_1} + z_2^{p_2} + \cdots + z_{n-3}^{p_{n-3}} + 
z_{n-2}^6 + z_{n-1}^6 + z_n^6 + z_{n+1}^6$$
is a square. This means that the characteristic
polynomials $\Delta_f(t)$ and $\Delta_g(t)$ of the
algebraic knots $K_f$ and $K_g$, respectively, satisfy
the Fox-Milnor type relation, although their exponents
are distinct. Thus the assumptions in 
Theorem~\ref{thm2} and Proposition~\ref{prop:multi}
are necessary, as long as the proof depends only
on the Fox-Milnor type relation.
\end{example}

\section{Further results}\label{section4}

In this section, we give some more precise
results for the case of two or three variables.

\begin{proposition}\label{prop:B2}
Let $f(z) = z_1^{a_1} + z_2^{a_2}$ and
$g(z) = z_1^{b_1} + z_2^{b_2}$ be Brieskorn polynomials
of two variables. If the Seifert forms $L_f$
and $L_g$ are Witt equivalent over the real
numbers, then $a_j = b_j$, $j = 1, 2$, up to order.
\end{proposition}

\begin{proof}
If $a_1$ or $a_2$ is odd, then by
Proposition~\ref{proposition:odd} we may assume
that $a_1 = b_1$ is odd.
Then by Proposition~\ref{prop:cot}, we have
$$\cot \frac{\pi}{2 a_2} = \cot \frac{\pi}{2 b_2},$$
which implies that $a_2 = b_2$.

Therefore, we may assume that all the exponents
for $f$ and $g$ are even. Then by Proposition~\ref{prop:FM} (2),
we have
$$(\Lambda_{a_1} - 1)(\Lambda_{a_2} -1)
\equiv (\Lambda_{b_1} - 1)(\Lambda_{b_2} -1) \pmod{2},$$
which implies that 
$$\Lambda_{a_1} + \Lambda_{a_2} \equiv
\Lambda_{b_1} + \Lambda_{b_2} \pmod{2}.$$
If $a_1 \neq a_2$, then we see that
$b_1 \neq b_2$, and $a_j = b_j$, $j = 1, 2$, up to order
by Lemma~\ref{lemma:lambda}.
If $a_1 = a_2$, then we must have $b_1 = b_2$.
In this case, by Proposition~\ref{prop:cot}, we have
$$\cot^2 \frac{\pi}{2 a_1} = \cot^2 \frac{\pi}{2 b_1},$$
which implies that $a_1 = b_1$.
This completes the proof.
\end{proof}

\begin{proposition}
Let $f(z) = z_1^{a_1} + z_2^{a_2} + z_3^{a_3}$ and
$g(z) = z_1^{b_1} + z_2^{b_2} + z_3^{b_3}$ 
be Brieskorn polynomials
of three variables. If the Seifert forms $L_f$
and $L_g$ are Witt equivalent over the real
numbers, then $a_j = b_j$, $j = 1, 2, 3$, up to order.
\end{proposition}

\begin{proof}
First suppose that $a_1$, $a_2$ and $a_3$ are all
even. Then by Proposition~\ref{proposition:odd},
$b_1$, $b_2$ and $b_3$ are all even. In this case,
by Proposition~\ref{prop:FM} (2), we have
$$\Lambda_{a_1} + \Lambda_{a_2} + \Lambda_{a_3}
\equiv \Lambda_{b_1} + \Lambda_{b_2} + \Lambda_{b_3}
\pmod{2}.$$
Thus, we may assume that $a_1 = b_1$ 
by Lemma~\ref{lemma:lambda}.
Then by Proposition~\ref{prop:cot}, we have
$$\cot \frac{\pi \ell}{2 a_2} \cot \frac{\pi \ell}{2 a_3}
= \cot \frac{\pi \ell}{2 b_2} \cot \frac{\pi \ell}{2 b_3}$$
for all odd integers $\ell$. Then, by Proposition~\ref{prop:B2},
we see that $a_j = b_j$, $j = 1, 2, 3$, up to order.

Now suppose that $a_1$, $a_2$ or $a_3$ is odd.
Then, by Proposition~\ref{proposition:odd}, we may assume
that $a_1 = b_1$ is odd and $a_2 \leq a_3$ and
$b_2 \leq b_3$. 

Then by Proposition~\ref{prop:cot}, we have
\begin{equation}
\cot \frac{\ell\pi}{2a_2} \cot \frac{\ell \pi}{2a_3}
= \cot \frac{\ell\pi}{2b_2} \cot \frac{\ell \pi}{2b_3}
\label{eq:3var}
\end{equation}
for all odd integers $\ell$ that are not a multiple of $a_1 = b_1$.
If $a_2 = b_2$, then putting $\ell = 1$, we get
$a_3 = b_3$. So, suppose that $a_2 < b_2$.
Then by (\ref{eq:3var}) with $\ell = 1$,
we have $a_2 < b_2 \leq b_3 < a_3$.

Let us consider the characteristic polynomials
$\Delta_f(t)$ and $\Delta_g(t)$. We have
\begin{eqnarray*}
\mathrm{divisor}\,\Delta_f(t) & = & (\Lambda_{a_1} -1)
(\Lambda_{a_2} - 1)(\Lambda_{a_3} - 1) \\
& = & (a_1, a_2)([a_1, a_2], a_3)\Lambda_{[a_1, a_2, a_3]}
- (a_1, a_2)\Lambda_{[a_1, a_2]} - (a_1, a_3)\Lambda_{[a_1, a_3]} \\
& & \quad - (a_2, a_3)\Lambda_{[a_2, a_3]} + \Lambda_{a_1}
+ \Lambda_{a_2} + \Lambda_{a_3} - 1
\end{eqnarray*}
and
\begin{eqnarray*}
\mathrm{divisor}\,\Delta_g(t) & = & 
(b_1, b_2)([b_1, b_2], b_3)\Lambda_{[b_1, b_2, b_3]}
- (b_1, b_2)\Lambda_{[b_1, b_2]} - (b_1, b_3)\Lambda_{[b_1, b_3]} \\
& & \quad - (b_2, b_3)\Lambda_{[b_2, b_3]} + \Lambda_{b_1}
+ \Lambda_{b_2} + \Lambda_{b_3} - 1.
\end{eqnarray*}
Since $[a_1, a_2, a_3]$, $[a_1, a_3]$,
$[a_2, a_3]$, $a_3$, 
$[b_1, b_2, b_3]$, $[b_1, b_2]$, $[b_1, b_3]$,
$[b_2, b_3]$, $b_2$ and $b_3$
are all strictly greater than $a_2$,
by Proposition~\ref{prop:FM} (2)
together with $a_1 = b_1$, we must have
$[a_1, a_2] = a_2$. Thus $a_2$ is a multiple of $a_1$.
Then by Proposition~\ref{prop:FM} (2) again, we have
\begin{eqnarray*}
\Lambda_{[a_1, a_3]} + \Lambda_{a_3}
& \equiv & ([b_1, b_2], b_3)\Lambda_{[b_1, b_2, b_3]}
+ \Lambda_{[b_1, b_2]} + \Lambda_{[b_1, b_3]} \\
& & \quad + (b_2, b_3)\Lambda_{[b_2, b_3]} + \Lambda_{b_2}
+ \Lambda_{b_3} \pmod{2},
\end{eqnarray*}
since $a_1 = b_1$ is odd.

If $b_2 < b_3$, then we must have
$[b_1, b_2] = b_2$, i.e., $b_2$ is a multiple of $b_1$.
Then, we see that $[a_1, a_3] = a_3$ and
$[b_1, b_3] = b_3$. Therefore,
$a_2$, $a_3$, $b_2$ and $b_3$ are all multiples of $a_1 = b_1$.
Since $a_1$ is odd and $a_1 \geq 3$,
there exists an odd integer $\ell$ ($= a_2+1$
or $a_2+2$)
which is not a multiple of $a_1$ such that
$a_2 < \ell < b_2$.
Then for this $\ell$,
the left hand side of (\ref{eq:3var})
is negative, while the right hand side is positive.
This is a contradiction.

If $b_2 = b_3$, then we have
$$\Lambda_{[a_1, a_3]} + \Lambda_{a_3}
\equiv b_2\Lambda_{[b_1, b_2]}
+ b_2\Lambda_{b_2} \pmod{2}.$$
Thus, $[a_1, a_3] = a_3$, and $a_3$
is a multiple of $a_1$.
Then, using an odd integer $\ell$ ($= a_2+1$ or $a_2+2$)
which is not a multiple of $a_1$ such that
$a_2 < \ell < a_3$ in (\ref{eq:3var}), we again
get a contradiction, since $b_2 = b_3$.

Therefore, we must have $a_2 = b_2$ and
$a_3 = b_3$. This completes the proof.
\end{proof}

\begin{proposition}\label{prop:whp2}
Let $f$ and $g$ be weighted homogeneous
polynomials of two variables with
weights $(w_1, w_2)$ and $(w_1', w_2')$,
respectively, with $w_j, w_j' \geq 2$.
If their Seifert forms
are Witt equivalent over the real numbers,
then $w_j = w_j'$, $j = 1, 2$, up to order.
\end{proposition}

\begin{proof}
Set $w_j = u_j/v_j$ and $w_j' = u_j'/v_j'$,
$j = 1, 2$, where $u_j$ and $v_j$
(resp.\ $u_j'$ and $v_j'$) are relatively
prime positive integers.
Let $m$ be a common multiple of $u_1$, $u_2$, $u'_1$ and
$u_2'$. Then, by the same argument as
in the proof of \cite[Lemma~3.1]{Saeki2000},
we see that the polynomial
\begin{eqnarray*}
h(\eta) & = & -\eta^{m/w_1 + m/w_2 + m/w_1'}
-\eta^{m/w_1 + m/w_2 + m/w_2'} \\
& & \quad 
+\eta^{m/w_1 + m/w'_1 + m/w_2'}
+\eta^{m/w_2 + m/w'_1 + m/w_2'} \\
& & \quad + \eta^{m/w_1} + \eta^{m/w_2}
- \eta^{m/w'_1} - \eta^{m/w'_2}
\end{eqnarray*}
in $\eta$ is divisible by $\eta^m + 1$.

Since
$$\cot \frac{\pi}{2w_1} \cot \frac{\pi}{2w_2}
= \cot \frac{\pi}{2w'_1} \cot \frac{\pi}{2w'_2},$$
we may assume that $w_1 \geq w_1' \geq w_2' \geq w_2$.
Furthermore, if $w_1 = w_1'$, then we have $w_2 = w'_2$.
Therefore, we may assume
$$w_1 > w_1' \geq w_2' > w_2 (\geq 2).$$

If
\begin{equation}
\frac{1}{w_2} + \frac{1}{w'_1} + \frac{1}{w_2'}
- \frac{1}{w_1} < 1,
\label{eq:small}
\end{equation}
then by the same argument as in 
the proof of \cite[Lemma~3.1]{Saeki2000},
we have the desired conclusion.

If (\ref{eq:small}) does not hold, then we have
\begin{eqnarray*}
h(\eta)
& = &
-\eta^{m/w_1 + m/w_2 + m/w_1'}
-\eta^{m/w_1 + m/w_2 + m/w_2'} \\
& & \quad 
+\eta^{m/w_1 + m/w'_1 + m/w_2'}
+(\eta^m +1)\eta^{m/w_2 + m/w'_1 + m/w_2'-m} \\
& & \quad -\eta^{m/w_2 + m/w'_1 + m/w_2'-m} 
+ \eta^{m/w_1} + \eta^{m/w_2} \\
& & \quad - \eta^{m/w'_1} - \eta^{m/w'_2}.
\end{eqnarray*}
Since $m/w_2 + m/w'_1 + m/w_2'-m < m$,
in order that $h(\eta)$ be divisible by $\eta^m+1$,
we must have that $1/w_2 + 1/w'_1 + 1/w'_2 - 1$
is equal to $1/w_1 + 1/w_1' + 1/w_2'$, $1/w_1$ or
$1/w_2$. The first case does not occur, since
$w_1 > w_2 \geq 2$. In the third case, we
have $1/w'_1 + 1/w'_2 = 1$, which implies that
$w_1' = w_2' = 2$. This is a contradiction,
since $w'_2 > w_2 \geq 2$.
In the second case, we have
\begin{eqnarray*}
h(\eta)
& \equiv &
-\eta^{m/w_1 + m/w_2 + m/w_1'}
-\eta^{m/w_1 + m/w_2 + m/w_2'} \\
& & \quad 
+\eta^{m/w_1 + m/w'_1 + m/w_2'} + \eta^{m/w_2} \\
& & \quad - \eta^{m/w'_1} - \eta^{m/w'_2}
\mod \eta^m+1.
\end{eqnarray*}
Then the difference of the highest degree and the
lowest one of the right hand side is equal to
$$(m/w_1 + m/w_2 + m/w_2') - m/w_1',$$
which is strictly less than $m$, since
$1/w_2 + 1/w_2' = 1/w_1 - 1/w'_1 + 1$.
This is a contradiction.

Therefore, we must have $w_1 = w_1'$ and $w_2 = w_2'$.
This completes the proof.
\end{proof}

By using exactly the same argument as 
in \cite[Lemma~3.1]{Saeki2000},
we have the following.

\begin{proposition}
Let $f$ and $g$ be nondegenerate weighted homogeneous 
polynomials in $\C^{n+1}$ with weights
$(w_1, w_2, \ldots, w_{n+1})$ and
$(w_1', w_2', \ldots, w_{n+1}')$, respectively,
such that $w_j \geq 2$ and $w'_j \geq 2$ for all $j$.
Suppose that the Seifert forms of $f$ and $g$
are Witt equivalent over the real
numbers. 
If
$$\sum_{j=1}^{n+1}\frac{1}{w_j} +
\sum_{j=1}^{n+1}\frac{1}{w_j'} -
2 \min\left\{\frac{1}{w_1}, \ldots, \frac{1}{w_{n+1}},
\frac{1}{w_1'}, \ldots, \frac{1}{w_{n+1}'}\right\} < 1,$$
then we have $w_j = w_j'$, $j = 1, 2, \ldots, n+1$,
up to order.
\end{proposition}

\begin{remark}
By Proposition~\ref{prop:whp2},
we see that if the algebraic knots
associated with two weighted homogeneous
polynomials of two variables
are cobordant, then the polynomials have
the same set of weights. In fact, this fact
itself is a consequence of already known results
as follows.

If two algebraic knots in $S^3$ are
cobordant, then they are in fact
isotopic by virtue of the results of
L\^e \cite{Le} and Zariski \cite{Z}
(for details, see \cite[\S4]{BS}).
Then, by Yoshinaga-Suzuki \cite{YS2}
(see also \cite{Kang, Ni}), they have the
same set of weights.
\end{remark}


\begin{thebibliography}{99999}
%
\bibitem{BM}V.~Blanl\oe il and F.~Michel,  
{\em A theory of cobordism for non-spherical links}, 
Comment.\ Math.\ Helv.\ \textbf{72} (1997), 30--51.  
%
\bibitem{BS}V.~Blanl\oe il and O.~Saeki, 
{\em Cobordism of fibered knots and related topics},
in ^^ ^^ Singularities in geometry and topology 2004", 
pp.~1--47, Adv.\ Stud.\ Pure Math.\ \textbf{46}, 
Math.\ Soc.\ Japan, Tokyo, 2007.
%
\bibitem{DH}P.~Du Bois and O.~Hunault, 
{\em Classification des formes de Seifert rationnelles 
des germes de courbe plane},
Ann.\ Inst.\ Fourier (Grenoble) \textbf{46} (1996), 371--410.
%
\bibitem{Durfee}A.~Durfee, 
{\em Fibered knots and algebraic 
singularities}, 
Topology \textbf{13} (1974), 47--59.
%
\bibitem{Kang}C.~Kang, 
{\em Analytic types of plane curve singularities 
defined by weighted homogeneous polynomials},
Trans.\ Amer.\ Math.\ Soc.\ \textbf{352} (2000), 3995--4006. 
%
\bibitem{Kato}M.~Kato,
{\em A classification of simple
spinnable structures on a $1$-connected Alexander manifold},
J.\ Math.\ Soc.\ Japan \textbf{26} (1974), 454--463.
%
\bibitem{King}H.C.~King, {\em Topological type of isolated 
critical points}, Ann.\ of Math.\ (2) \textbf{107} (1978), 385--397.
%
\bibitem{Le}D.~T.~L\^e, 
{\em Sur les n\oe uds alg\'ebriques},
Compositio Math.\ \textbf{25} (1972), 281--321. 
%
\bibitem{L1}J.~Levine, 
{\em Knot cobordism groups in 
codimension two}, 
Comment.\ Math.\ Helv.\ \textbf{44} (1969),  
229--244.
%
\bibitem{L2}J.~Levine,
{\em Invariants of knot cobordism},
Invent.\ Math.\ \textbf{8} (1969), 98--110; 
addendum, ibid.\ \textbf{8} (1969), 355.
%
\bibitem{Milnor}J.~Milnor, 
{\sl Singular points of complex hypersurfaces}, 
Ann.\ of Math.\ Stud., Vol.~61, 
Princeton Univ.\ Press, Princeton, N.J.; Univ.\ of Tokyo Press, 
Tokyo, 1968. 
%
\bibitem{MO}J.~Milnor and P.~Orlik,
{\em Isolated singularities defined by weighted 
homogeneous polynomials},
Topology \textbf{9} (1970), 385--393. 
%
\bibitem{N}A.~N\'{e}methi, {\em The real Seifert form and 
the spectral pairs of isolated hypersurface singularities},
Compositio Math.\ \textbf{98} (1995), 23--41.
%
\bibitem{Neumann}W.D.~Neumann, 
{\em Invariants of plane curve singularities},
N\oe uds, tresses et singularit\'e
(Plans-sur-Bex, 1982), pp.~223--232, 
Monogr.\ Enseign.\ Math., Vol.~31, Enseignement Math., Geneva, 1983. 
%
\bibitem{Ni}T.~Nishimura,
{\em Topological invariance of weights for 
weighted homogeneous singularities},
Kodai Math.\ J.\ \textbf{9} (1986), 188--190. 
%
\bibitem{Perron}B.~Perron,
{\em Conjugaison topologique des germes de fonctions holomorphes \`a 
singularit\'e isol\'ee en dimension trois},
Invent.\ Math.\ \textbf{82} (1985), 27--35. 
%
\bibitem{Saeki88}O.~Saeki, 
{\em Cobordism classification of knotted homology $3$-spheres 
in $S^5$}, 
Osaka J.\ Math.\ \textbf{25} (1988), 213--222.
%
\bibitem{Saeki89}O.~Saeki,
{\em Topological types of complex isolated hypersurface singularities},
Kodai Math.\ J.\ \textbf{12} (1989), 23--29. 
%
\bibitem{Saeki2000}O.~Saeki, 
{\em Real Seifert form determines the spectrum for 
semiquasihomogeneous hypersurface singularities in $\C^3$},
J.\ Math.\ Soc.\ Japan \textbf{52} (2000), 409--431. 
%
\bibitem{Saito}K.~Saito,
{\em Quasihomogene isolierte Singularit\"{a}ten von 
Hyperfl\"{a}chen},
Invent.\ Math.\ \textbf{14} (1971), 123--142. 
%
\bibitem{Sakamoto}K.~Sakamoto,
{\em The Seifert matrices of Milnor fiberings 
defined by holomorphic functions},
J.\ Math.\ Soc.\ Japan \textbf{26} (1974), 714--721.
%
\bibitem{SSS}R.~Schrauwen, J.~Steenbrink, and J.~Stevens,
{\em Spectral pairs and the topology of curve singularities},
Complex geometry and Lie theory (Sundance, UT, 1989), pp.~305--328, 
Proc.\ Sympos.\ Pure Math., Vol.~53, Amer.\ Math.\ Soc., 
Providence, RI, 1991. 
%
\bibitem{Steenbrink}J.~H.~M.~Steenbrink, 
{\em Mixed Hodge structure on the vanishing cohomology}, 
in ``Real and complex singularities (P.~Holm, ed.)", 
Stijthoff-Noordhoff, Alphen a/d Rijn, 1977, pp.~525--563.
%
\bibitem{Steenbrink2}J.~H.~M.~Steenbrink,
{\em Intersection form
for quasihomogeneous singularities}, Compositio Math.\ 
\textbf{34} (1977), 211--223.
%
\bibitem{YS}E.~Yoshinaga and M.~Suzuki,
{\em On the topological types of singularities of 
Brieskorn-Pham type},
Sci.\ Rep.\ Yokohama Nat.\ Univ.\ Sect.\ I \textbf{25}
(1978), 37--43.
%
\bibitem{YS2}E.~Yoshinaga and M.~Suzuki, 
{\em Topological types of quasihomogeneous singularities 
in $\C^2$}, Topology \textbf{18} (1979), 113--116.
%
\bibitem{Z}O.~Zariski, 
{\em On the topology of algebroid singularities},
Amer.\ J.\ Math.\ \textbf{54} (1932), 453--465. 

\end{thebibliography}
\end{document}